\documentclass[10pt]{amsart}
\usepackage{amssymb}
\usepackage[all]{xy}

 \theoremstyle{plain}

\newtheorem{thm}{Theorem}[section]
\newtheorem{lemma}[thm]{Lemma}
\newtheorem{prop}[thm]{Proposition}
\newtheorem{cor}[thm]{Corollary}

\newtheorem*{thmA}{Theorem A}
\newtheorem*{thmB}{Theorem B}
\newtheorem*{thmC}{Theorem C}
\newtheorem*{thmD}{Theorem D}
\newtheorem*{thmE}{Theorem E}

\theoremstyle{definition}

\newtheorem{rem}[thm]{Remark}

\newtheorem*{quesF}{Question F}

\newtheorem{example}{Example}[section]

\numberwithin{equation}{section}

\newcommand{\F}{\mathbb{F}}
\newcommand{\N}{\mathbb{N}}
\newcommand{\Z}{\mathbb{Z}}
\newcommand{\Q}{\mathbb{Q}}

\newcommand{\bE}{\mathbb{E}}
\newcommand{\bbE}{\bar{\bE}}
\newcommand{\bQ}{\bar{\Q}}

\newcommand{\boC}{\mathbf{C}}

\newcommand{\tphi}{\tilde{\phi}}
\newcommand{\bchi}{\bar{\chi}}

\DeclareMathOperator{\Zen}{Z}
\DeclareMathOperator{\SL}{SL}
\DeclareMathOperator{\GL}{GL}
\DeclareMathOperator{\Syl}{Syl}
\DeclareMathOperator{\Cent}{Cent}
\DeclareMathOperator{\kernel}{ker}
\DeclareMathOperator{\image}{im}
\DeclareMathOperator{\rst}{res}
\DeclareMathOperator{\mor}{mor}

\DeclareMathOperator{\gr}{gr}

\DeclareMathOperator{\Ann}{Ann}
\DeclareMathOperator{\End}{End}
\DeclareMathOperator{\PL}{PL}
\DeclareMathOperator{\iid}{id}
\DeclareMathOperator{\incl}{int}
\DeclareMathOperator{\Jac}{Jac}
\DeclareMathOperator{\ssl}{ss}
\DeclareMathOperator{\un}{un}

\DeclareMathOperator{\cl}{cl}

\newcommand{\ca}[1]{\mathcal{#1}}

\newcommand{\argu}{\hbox to 7truept{\hrulefill}}

\title{Finite p-central groups of height k}
\author{J. Gonz\'alez-S\'anchez 
and T. S. Weigel}
\address{T.~S.~Weigel\\
Universit\`a di Milano-Bicocca\\
U5-3067, Via R.Cozzi, 53\\
20125 Milano, Italy}

\email{thomas.weigel@unimib.it}

\address{Jon Gonz\'alez-S\'anchez\\
Universidad de Cantabria\\
Departamento de Matem\'aticas, Estad\'{\i}stica y Computaci\'on \\
Facultad de Ciencias, Avda. de los Castros, s/n \\
E-39071 Santander, Spain}

\email{jon.gonzalez@unican.es}

\thanks{The first author acknowledges support by the 
Spanish Ministerio de Ciencia e Innovaci\'on, grant MTM2008-06680-C02-01.}

\date{\today}

\begin{document}

\begin{abstract}
A finite group $G$ is called {\it $p^i$-central of height $k$} if every element of order
$p^i$ of $G$ is contained in the $k^{th}$-term $\zeta_k(G)$ of the ascending central series of $G$. If $p$ is odd such a group has to be $p$-nilpotent (Thm.~A).
Finite $p$-central $p$-groups of height $p-2$ can be seen as the dual analogue
of finite potent $p$-groups, i.e., for such a finite $p$-group $P$ the group
$P/\Omega_1(P)$ is also $p$-central of height $p-2$ (Thm.~B).
In such a group $P$ the index of $P^p$  is less or equal than
the order of the subgroup $\Omega_1(P)$ (Thm.~C).
If the Sylow $p$-subgroup $P$ of a finite group $G$
is $p$-central of height $p-1$, $p$ odd, and $N_G(P)$ 
is $p$-nilpotent, then $G$ is also $p$-nilpotent (Thm.~D).
Moreover, if $G$ is a $p$-soluble finite group, $p$ odd, and
$P\in \Syl_p(G)$ is $p$-central of height $p-2$, then $N_G(P)$ controls $p$-fusion
in $G$ (Thm.~E).
It is well-known that the last two properties hold for Swan groups (see \cite{hp:pcen}).
\end{abstract}

\maketitle

\section{Introduction}
\label{s:intro}
Let $p$ be a prime number. 
Following a suggestion of A.~Mann\footnote{A.~Mann suggested this terminology to the second author during a train journey from Dublin to Galway in 1994.} we call a finite group $G$ {\it $p$-central} if every element of order $p$ is central
in $G$, i.e., let
\begin{equation}
\label{eq:defOm}
\Omega_i(G,p)=\langle\,g\in G\mid g^{p^i}=1\,\rangle,\qquad i\geq 0;
\end{equation}
then $G$ is $p$-central if, and only if, $\Omega_1(G,p)\leq\Zen(G)$. 
For a finite $p$-group $P$
we put $\Omega_i(P)=\Omega_i(P,p)$.
Obviously, subgroups of $p$-central groups are $p$-central.

Powerful finite $p$-groups play an important role in the study of compact $p$-adic analytic groups (see \cite[\S 3]{ddms:padic}). 
For finite $p$-central $p$-groups the notion of $p$-centrality is
the dual analogue of being powerful provided $p$ is odd.
Although one must say that $p$-centrality plays a less important role
in the study of pro-$p$ groups,
finite $p$-central groups are quite often the most important examples
when it comes to explicit computations in cohomology or homotopy of classifying spaces of finite groups (see for example  \cite{brothenn:cm}, 
\cite[Thm.~1.7]{hp:pcen}, \cite{th:pcenp}). 

In this note we study a more general class of finite
groups. For $i,k\geq 1$ we say that a finite group $G$ is {\it $p^i$-central
of height $k$} if $\Omega_i(G,p)\leq \zeta_k(G)$, where $\zeta_k(G)$ denotes the
$k^{th}$-term of the ascending central series of $G$. In particular, if $G$ is a 
finite $p^i$-central group of height $k$, then $\Omega_i(G,p)$ is a $p$-group.
If $G$ is a $p^i$-central group of height 1 we simply call $G$ {\it $p^i$-central}.
For odd primes $p$ one can restrict the analysis of finite $p$-central 
groups to the study of $p$-central $p$-groups (see \S\ref{ss:finpcen}).

\begin{thmA}
Let $p$ be odd, and let $G$ be a finite $p$-central group of height $k\geq 1$.
Then $G$ is $p$-nilpotent, i.e., if $P\in\Syl_p(G)$ is a Sylow $p$-subgroup of $G$,
then $P$ has a normal complement in $G$.
\end{thmA}

For $p$-central finite groups Theorem~A was already proved in 
\cite[Thm.~2.4]{th:pcenp}.

\begin{example}
\label{ex:2cen}
For an odd prime number $r$ the finite group $G=\SL_2(r)$ is $2$-central. Hence Theorem~A does not hold for $p=2$.
\end{example}

The first remarkable result for finite $p$-central $p$-groups is the fact 
that for a finite $p$-central $p$-group $P$, $p$ odd, 
$P/\Omega_i(P)$ is also $p$-central, i.e.,
$(\Omega_i(P))_{i\geq1}$ is an ascending central series of $P$.
In this form one can find this result in a paper by J.~Buckley (see \cite{buck:pcen}).
However, in his first book which appeared three years earlier B.~Huppert
cited an unpublished result of J.~G.~Thompson (see \cite[Hilfssatz~12.2]{hupp1:grp}) 
reproducing a proof he is giving credit to N.~Blackburn which is also implying the
just mentioned statement.
For this reason we refer to this result as the
Blackburn-Buckley-Thompson theorem.
For $p=2$ it remains true if one requires that $\Omega_2(P)$
is contained in the center of $P$ (see \cite{mann:4cen}).
In Section~\ref{ss:hom} we will prove the following generalization.

\begin{thmB}
Let $k$ and $i$ be positive integers such that $k\leq p-2$ or $k=p-1$
and $i\geq 2$. Let $P$ be a $p^i$-central pro-$p$ group of height
$k$ such that the torsion elements of $P$ form a finite $p$-group. Then,
for all $j\geq 1$, $P/\Omega_j(P)$ is also a $p^i$-central pro-$p$
group of height $k$. Moreover 
\begin{equation}
\label{eq:omega}
\Omega_j(P)=\{ x\in P\mid x^{p^j}=1\}.
\end{equation}
\end{thmB}

Some families of finite $p$-groups - like powerful or regular groups - 
have a nice power structure. 
In particular, the index of the agemo subgroup (the subgroup being
generated by the $p$-powers or, simply, $P^p$) is equal to the number of elements of order $p$ (see \cite{gust:omegapo}, \cite{hall:pgr}). 
For an arbitrary finite $p$-group $P$ the index $|P:P^p|$ is bounded in terms
of $p$ and  $|\Omega_1(P)|$, i.e., there exist a function 
$f$ such that $|P:P^p|\leq f(p,|\Omega_1(P)|)$ 
(see \cite{jon:agemo}, \cite{mann:pow}). 
In a finite $p$-central $p$-group $P$ of height $p-2$ the index of the agemo subgroup is bounded just by the cardinality of $|\Omega_1(P)|$
(see \S\ref{ss:agemo}).

\begin{thmC}
Let $k,i$ be positive integers such that $k\leq p-2$
or $k=p-1$ and $i\geq 2$, and let $P$ be a $p^i$-central $p$-group of height $k$. 
Then
\begin{equation}
\label{eq:ineq}
|P:P^p|\leq
|\Omega_1(P)|.
\end{equation} 
Moreover, if $P^p=\{x^p\mid x\in P\}$, then equality holds in \eqref{eq:ineq}.
\end{thmC}

In \cite{hp:pcen}, S.~Priddy and H-W.~Henn showed
that a finite $p$-central $p$-group $P$, $p$ odd\footnote{For $p=2$ one has to require
that $(\Omega_i(P))_{i\geq 1}$ is an ascending central series.}, is a {\it Swan group}, 
i.e.,
if $G$ is a finite group containing $P$ as Sylow $p$-subgroup, then
the restriction map in cohomology
\begin{equation}
\label{eq:swan}
\rst^G_{N_G(P)}\colon H^\bullet(G,\F_p)\longrightarrow H^\bullet(N_G(P),\F_p)
\end{equation}
is an isomorphism of $\N_0$-graded, connected, anti-commutative $\F_p$-algebras.
As pointed out by J.~Th\'evenaz in \cite{thev:most} 
a finite $p$-group $P$ is a Swan group if, and only if, $N_G(P)$ controls
$p$-fusion in $G$ for all finite groups $G$ with $P\in\Syl_p(G)$.
This property of $p$-central groups seems to be very particular, and one cannot expect
that it holds in a much more general context (see Ex.~\ref{ex:psl2}).
Nevertheless, there are two weaker properties which also hold for
$p$-central groups of small height.

\begin{example}
\label{ex:psl2}
Let $p\geq 5$, $n\geq 2$ and put $G=\SL_2(\Z/p^n.\Z)$.
Then $P\in\Syl_p(G)$ is $p$-central of height $3$, but
$N_G(P)$ does not control $p$-fusion in $G$.
Hence, by \cite{mis:modp}, $P$ is not a Swan group.
\end{example}

A finite $p$-group $P$ will be called to {\it determine $p$-nilpotency locally}
if the following holds: Let $G$ be a finite group containing $P$ as Sylow $p$-subgroup
such that $N_G(P)$ is $p$-nilpotent. Then $G$ is also $p$-nilpotent.
From J.~Tate's nilpotency criterion (see \cite{tat:nil}) it
follows that Swan groups determine $p$-nilpotency locally.
For finite $p$-central $p$-groups of small height we will show the following (see \S\ref{s:dnl}).

\begin{thmD}
Let $p$ be odd, and let $P$ be a finite $p$-central $p$-group of height $p-1$.
Then $P$ determines $p$-nilpotency locally.
\end{thmD}

Finite $p$-central $p$-groups of small height possess a `Swan group like' property
which concerns their embeddings into $p$-soluble groups.
For simplicity we say that a $p$-soluble group $G$ is a
{\it $p^\prime p p^\prime$-sandwich group}, if
$O_p(G/O_{p^\prime}(G))$ is a Sylow $p$-subgroup of $G/O_{p^\prime}(G)$.
In Section~\ref{ss:sand} we will show the following.

\begin{thmE}
Let $p$ be odd, and let $G$ be a finite $p$-soluble group
containing a $p$-central Sylow $p$-subgroup $P$ of height $p-2$.
Then $G$ is a $p^\prime p p^\prime$-sandwich group.
In particular, $N_G(P)$ controls $p$-fusion in $G$.
\end{thmE}

The mod $p$ cohomology ring $H^\bullet(G)=H^\bullet(G,\F_p)$ of a finite $p$-central group $G$ is much easier to analyze than for an arbitrary finite group.
In  \cite{brothenn:cm}, C.~Broto and H-W.~Henn showed that 
$H^\bullet(G)$ is Cohen-Macaulay.
Thus by a result of
D.~J.~Benson and J.~F.~Carlson (see \cite{bencar:poin}) its Hilbert series
\begin{equation}
\label{eq:hilb}
h_G(t)=\textstyle{\sum_{k\in\N_0} \dim_{\F_p}(H^k(G)). t^k}
\end{equation}
satisfies some functional equation. In case that $p$ is odd and
$G$ satisfies some additional property - the {\it $\Omega_1$-extension property}
(see Remark~\ref{rem:o1ep}) -
one can determine the $\F_p$-algebra structure of $H^\bullet(G)$
explicitly modulo some finite-dimensional $\F_p$-algebra which satisfies
Poincar\'e duality (see \cite{th:pcenp}). 

A finite $p$-central $p$-group $P$ of height $2$ may have maximal elementary
abelian subgroups of different order, and thus, its mod $p$-cohomology ring will not be Cohen-Macaulay (see \cite{bencar:poin}).
In order to know whether one can extend the result of 
C.~Broto and H-W.~Henn to $p$-central $p$-groups of height $p-2$,
it would be interesting to investigate the following question.

\begin{quesF}
Let $p$ be odd, and let $P$ be a finite
$p$-central $p$-group of height $p-2$ such that $\Omega_k(P)/\Omega_{k-1}(P)$
is abelian for all $k\geq 1$.
Is $H^\bullet(P)$ Cohen-Macaulay?
\end{quesF}

{\bf Acknowledgement.}
The authors would like to thank A.~E.~Zalesskii for some very helpful discussions.
\vskip12pt 

{\it Notations and basic definitions.} We use the standard
notations in group theory. Multiple commutators will always supposed to be left normed.
We simply write $[N,_{(k)}M]$ for the commutator
$[N,M,\ldots ,M]$ with $M$ appearing $k$-times. 
By $G^{p^j}$ we denote the (closed) subgroup of the (topological) group $G$ which is
(topologically) generated by the $p^j$-powers of $G$; $\Omega_i(G,p)$ will
denote the (closed) subgroup generated by the elements of order less or equal to $p^i$. For a (closed) normal subgroup $N$ of $G$ we denote by $N^{1/p}$ the (closed)
normal subgroup generated by the set $\{ x\in G\mid x^p\in N\}$.

\section{$p$-Central $p$-groups of height $k$}
\label{s:pcenp}
Let $p$ be a prime number, and 
let $k$ be a positive integer satisfying $k\leq p-1$.
For a pro-$p$ group $P$ and an integer $n$ we put
\begin{equation}
\label{eq:lambdakf}
\lambda_n^{(k)}(P)=
\begin{cases}
\hfil P\hfil &\ \text{if $n\leq 1$,}\\
\cl(\lambda^{(k)}_{n-k}(P)^p.[\lambda^{(k)}_{n-1}(P),P])&\ \text{if $n>1$,}
\end{cases}
\end{equation}
where $\cl(\argu)$ denotes the closure operation.
One has the following.

\begin{prop}
\label{prop:pcenser}
Let $P$ be a pro-$p$ group, and let $k\leq p-1$. For $m,n\geq 1$ one has:
\begin{itemize}
\item[(a)]  $[\lambda_n^{(k)}(P),\lambda_m^{(k)}(P)]\leq\lambda_{m+n}^{(k)}(P)$;
\item[(b)]  $\lambda_n^{(k)}(P)^p\leq\lambda_{n+k}^{(k)}(P)$;
\item[(c)] if $k\leq p-2$, then for $x,y\in \lambda_n^{(k)}(P)$, $(xy)^p\lambda_{n+k+1}^{(k)}(P)=x^py^p\lambda_{n+k+1}^{(k)}(P)$.
\end{itemize}
In particular, $(\lambda_n^{(k)}(P))_{n\geq 1}$ is a descending $p$-central series.
If $k\leq p-2$ the associated graded object 
\begin{equation}
\label{eq:lamgr}
\gr^{(k)}_\bullet(P)=\textstyle{\coprod_{n\geq 1}} 
\lambda^{(k)}_n(P)/\lambda^{(k)}_{n+1}(P)
\end{equation}
carries naturally the structure of
a graded $\F_p[t_\bullet]$-Lie algebra with $t_\bullet$ 
being homogeneous of degree $k$.
\end{prop}

\begin{proof}
(a) and (b) follow from a similar argument as in \cite[Lemma 4.6]{gust:omega}. By 
\cite[Chap. III, Theorem 9.4]{hupp1:grp}, one has 
\begin{equation}
(xy)^p\equiv x^py^p\pmod{\gamma_2(\lambda_n^{(k)}(P))^p\gamma_p(\lambda_n^{(k)}(P))}.
\end{equation}
By (a) and (b), $\gamma_2(\lambda_n^{(k)}(P))^p\gamma_p(\lambda_n^{(k)}(P))\leq \lambda_{n+k+1}^{(k)}(P)$. This yields the claim.
\end{proof}

Note that for $k=1$ the series $(\lambda_n^{(k)}(P))_{n\geq 1}$
coincides with the descending $p$-central series of $P$.
A finite $p$-group $P$ is called of {\it $k$-elementary $p$-length $\ell$}
if $\gr^{(k)}_{\ell}(P)\not=0$, but $\gr^{(k)}_{\ell+1}(P)=0$. In particular,
a finite $p$-group $P$ is of
$k$-elementary $p$-length $\ell$, $\ell\leq k$, if, and only if, it is of exponent $p$
and nilpotency class $\ell$.

\subsection{$p$-Central $p$-groups of small height}
\label{ss:pcen1}
Let $p$ be odd,  let $k\leq p-2$, and let $P$ be a finite $p$-group.
Then $P$ is $k$-powerful if, and only if,
$t_\bullet\colon \gr^{(k)}_\bullet(P)\to\gr^{(k)}_{\bullet+k}(P)$ is surjective.
A weak form of injectivity of the mapping $t_\bullet$ has the following consequence.

\begin{prop}
\label{pro:inj}
Let $P$ be a finite $p$-group of $k$-elementary $p$-length $\ell$, $k\leq p-2$, and assume that
$t_s\colon \gr^{(k)}_s(P)\to\gr^{(k)}_{s+k}(P)$ is injective for $s<\ell-k+1$.
Then $P$ is $p$-central of height $k$.
\end{prop}

\begin{proof} By hypothesis, one has $\lambda^{(k)}_{\ell-k+1}(P)\leq \zeta_k(P)$,
and every element of order $p$ must be contained in $\lambda^{(k)}_{\ell-k+1}(P)$.
\end{proof}

\begin{example}
\label{ex:a} Let $k\leq p-2$.

\noindent
(a) Let $F$ be a finitely generated free pro-$p$ group. 
Then Proposition~\ref{pro:inj} implies that
for any
$m\geq 1$ the finite $p$-group $P_m=F/\lambda^{(k)}_{m+k}(F)$ is 
$p$-central of height $k$.

\noindent
(b) Let $Q$ be a finitely generated torsion-free $k$-powerful
pro-$p$ group. Then Proposition~\ref{pro:inj} shows that 
$Q/\lambda^{(k)}_{m+k}(Q)$ 
is a finite $p$-central $p$-group of height $k$. 
In particular, if $Q$ is a finitely generated uniformly powerful
pro-$p$ group, then $Q/Q^{p^{m+1}}$ is a finite $p$-central $p$-group.
This result can also be obtained from the theory of powerful groups.
(see \cite{ddms:padic}.)

\noindent
(c) Let $M$ be a finite $\Z_p$-module. It is straight forward to check that
\begin{equation}
\label{eq:expl}
\PL(M)=\{\, \iid_M+p.\alpha\mid \alpha\in\End_{\Z_p}(M)\,\}.
\end{equation}
is a $p$-central $p$-group. For $p$ odd it is also powerful.

\noindent
(d) For $n\geq 2$ let $P$ be a Sylow pro-$p$ subgroup of $\SL_n(\Z_p)$,
and for $m\geq 1$ let $K_m$ denote the kernel of the canonical homomorphism
$\SL_n(\Z_p)\to\SL_n(\Z_p/p^m.\Z_p)$, i.e., $K_m$ is just the $m^{th}$-congruence subgroup.

Suppose that $p>{n+1}$. Then $P$ is torsion-free. Moreover, one has
\begin{equation}
\label{eq:ex1}
\gamma_{n+1}(P)\leq P^p,
\end{equation}
i.e., $P$ is $n$-powerful and thus potent. This fact can be used to show that for
$P_m=P/K_{m+1}$ one has
\begin{equation}
\label{eq:exs}
\Omega_1(P_m)=K_m/K_{m+1}.
\end{equation}
In particular, $P_m$ is $p$-central of height $2n-1$.
\end{example}

\subsection{Homomorphic images of $p$-central pro-$p$ groups of height $k$}
\label{ss:hom}
The following power-commutator relations are well-known.

\begin{lemma}
\label{lemma:potom}
Let $P$ be a pro-$p$ group, and let $N$ be a closed normal subgroup of $P$.
If $N=\langle X \rangle$ then
\begin{equation}
\label{eq:potom}
[N,_{(r)} P]^p \leq [P (N),_{(r)}P][N,_{(r +p-1)} P],
\end{equation}
where $P(N)=\langle\, x^p\mid x\in X\, \rangle^G$.
In particular, for $i\geq 1$ one has
\begin{equation}
\label{eq:potom2}
[\Omega_i(P),_{(r)} P]^p \le [\Omega_{i-1}(P),_{(r)} P]\,[\Omega_{i}(P),_{(r+p-1)} P].
\end{equation}
\end{lemma}

\begin{proof}
See Theorem 2.7 of \cite{gust:omega}.
\end{proof}

Theorem B can be deduced from Lemma \ref{lemma:potom} as follows.

\begin{proof}[Proof of Theorem B]
By induction, it suffices to prove the theorem
for $j=1$. Let $T$ be the torsion subgroup. By hypothesis, $T$ is closed and
nilpotent. Moreover, there exists $m>0$ such that $[T,_{(m)}P]=1$.

As the nilpotence class of $\Omega_1(P)$ is bounded by $p-1$,
$\Omega_1(P)=\{\,x\in P\mid x^p=1\,\}$ (see \cite{hall:pgr}). 
In order to show that $P/\Omega_1(P)$
is $p^i$-central of height $k$ it suffices to show that for all $r\geq k$,
$[\Omega_{i+1}(P),_{(r)}P]\leq \Omega_1(P)$. 
As the claim certainly holds for $r\geq m$, we proceed by reverse
induction on $r< m$. 
Suppose that the assumption is true for $r+1\leq m$. 

\noindent
(1) If $k\leq p-2$, then by Lemma \ref{lemma:potom} and induction
\begin{equation}
\label{eq:prthmB}
[\Omega_{i+1}(P),_{(r)}P]^p 
\leq [\Omega_i(P),_{(r)}P][\Omega_{i+1}(P),_{(r+p-1)}P]\leq
[\Omega_1(P),_{(p-2)}P]=1.
\end{equation}
Hence $[\Omega_{i+1}(P),_{(r)}P]\leq\Omega_1(P)$.

\noindent 
(2) If $k=p-1$ and $i\geq 2$ the same argument as in \eqref{eq:prthmB} yields
\begin{equation}
\label{eq:prthmB2}
[\Omega_{i+1}(P),_{(r)}P]^p  
\leq\Omega_1(P)\cap \Zen(P).
\end{equation} 
Hence $[\Omega_{i+1}(P),_{(r)}P]\leq \Omega_2(P)$, and therefore
\begin{equation}
\label{eq:prthmB3}
[\Omega_{i+1}(P),_{(r)}P]^p\leq
 [\Omega_i(P),_{(r)}P][\Omega_{i+1}(P),_{(r+p-1)}P]\leq [\Omega_2(P),_{(p-1)}P]=1.
 \end{equation}
Thus $[\Omega_{i+1}(P),_{(r)}P]\leq \Omega_1(P)$, and this yields the claim.
\end{proof}

The following example shows that the hypothesis in Theorem B 
on the finiteness of the subgroup generated by torsion elements
cannot be removed.

\begin{example}
Let $p$ be odd, let $C_{p^{i+1}}$ be the cyclic group of order $p^{i+1}$, $i\geq 1$,
and let $A$ denote its canonical homomorphic image of order $p$.
Let $M=\Z_p[A]/\Z_p.\Delta$, where $\Delta=\sum_{a\in A} a$.
The semi-direct product $P=C_{p^{i+1}}\ltimes M$, with the canonical action of
$C_{p^{i+1}}$ on $M$, is a $p$-adic analytic pro-$p$ group which is $p^i$-central.
However, $P/\Omega_i(P)$ is not $p$-central.
If $c\in C_{p^{i+1}}$ is a generator and $a\in M$, then $c.a$ is an element of order
$p^{i+1}$. Hence the closed subgroup of $P$ generated by all torsion elements is
the entire group $P$.
\end{example}

Theorem B can be applied also to 
certain families of $p$-central pro-$p$ groups.

\begin{cor}
\label{cor:thmB}
Let $P$ be a $p^i$-central $p$-adic analytic pro-$p$ group of height $k$
such that $\gamma_{h(p-1)}(P)\leq \gamma_i(P)^{p^j}$ with
$h(p-1)<j(p-1)+i$. Then $P/\Omega_j(P)$ is also a $p^i$-central
pro-$p$ group of height $k$ and $\Omega_j(P)=\{\, x\in P\mid
x^{p^j}=1\,\}$.
\end{cor}

\begin{proof}
By  \cite[Thm.~4.8]{gust:omega}, the torsion elements of
$P$ form a finite $p$-group. Thus Theorem B yields the claim.
\end{proof}

\begin{rem}
\label{rem:o1ep}
Let $p$ be odd, and let $P$ be a $p$-central $p$-group of height $k$,
$k\leq p-2$. We say that $P$ has the {\it $\Omega_1$-extension property}
if there exists a finite $p$-central $p$-group $Q$ of height $k$ such that
$P\simeq Q/\Omega_1(Q)$.
It can be shown that all $p$-groups\footnote{In Example~\ref{ex:a}(d) one has to require that $n\leq (p-1)/2$.} 
described in Example \ref{ex:a} have this property.
However, there are examples of $p$-central groups which do not have this property.
For $p$-central $p$-groups of height $1$ this property has a strong impact
on the structure of the $\F_p$-cohomology algebra $H^\bullet(P,\F_p)$
(see \cite{th:pcenp}). It is very likely that this is also the case for
$p$-central $p$-groups of height larger than $1$ but less than $p-2$.
\end{rem}

\subsection{Bounding the index of the agemo subgroup}
\label{ss:agemo}
For a finite $p$-group $P$ we define the
following series of characteristic subgroups:
\begin{equation}
\label{eq:Mser}
M_n(P) =
\begin{cases}
[\Omega_1(P),_{(p-n-1)}P] & \text{if $n\leq p-1$,}
\\[8pt]
M_{n-p+1}(P)^{1/p} & \text{otherwise.}
\end{cases}
\end{equation}
It is easy to check that $M_{n+1}(P)\geq M_n(P)$. For $n<0$ we put $M_n(P)=1$.
If $P$ is a $p$-central $p$-group of 
height $p-2$ (resp. a $p^2$-central group of height $p-1$) this series has the following
properties.

\begin{prop}
\label{prop-1} 
Let $p$ be odd, let $P$ be a finite $p$-central $p$-group of height $p-2$
(resp. a $p^2$-central group of height $p-1$), and let $n\geq 0$. Then one has the following:
\begin{itemize}
\item[(a)] $P/M_n(P)$ is a $p$-central group of height $p-2$ (resp. a $p^2$-central group of height $p-1$).
\item[(b)] $[M_n(P),P]\leq M_{n-1}(P)$.
\item[(c)] $M_n(P)^p\leq M_{n-p+1}(P)$.
\item[(d)] If $x\in M_n(P)$, $y\in M_m(P)$ and $n\geq m$, then 
\begin{equation}
\label{eq:madd}
(xy)^p\equiv x^py^p\pmod{M_{m-p}(P)}.
\end{equation}
\end{itemize}
\end{prop}

\begin{proof}
(a) We restrict our considerations to the case when $P$ is $p$-central of height $p-2$.
The case when $P$ is $p^2$-central of height $p-1$ follows by a similar argument
(see proof of Theorem B).

By Theorem B, it suffices to prove the statement for
$n\leq p-1$. Hence one has to show that for $1\leq s\leq p-1$
and $r\geq p-2-s$, 
\begin{equation}
\label{eq:Ma}
[[\Omega_1(P),_{(s)}P]^{1/p},_{(r)}P]\leq
\Omega_1(P). 
\end{equation}
We proceed by reverse induction on $r$. For $r$ large, \eqref{eq:Ma} is certainly true. Suppose that \eqref{eq:Ma} is true for $r+1$. Then, by
Lemma~\ref{lemma:potom},
\begin{equation}
\begin{aligned}
{[[\Omega_1(P),_{(s)}P]}^{1/p},_{(r)}P]^p & \leq 
[\Omega_1(P),_{(s+r)}P][[\Omega_1(P),_{(s)}P]^{1/p},_{(r+p-1)}P] \\
& \leq  [\Omega_1(P),_{(p-2)}P]=1.
\end{aligned}
\end{equation}
Hence $[[\Omega_1(P),_{(s)}P]^{1/p},_{(r)}P]\leq \Omega_1(P)$, and this yields the claim.

\noindent (b) We may assume that $n\geq p-1$, and proceed by induction on $n$. 
By (a) and induction, one has
\begin{equation}
\label{eq:Mb}
\begin{aligned}
{[M_n(P),P]}^p & =  [M_{n-p+1}(P)^{1/p},P]^p\leq
[M_{n-p+1}(P),P][M_{n-p+1}(P)^{1/p},_{(p)}P] \\
& \leq M_{n-p}(P)\,[M_{n-p+1},P] =  M_{n-p}(P).
\end{aligned}
\end{equation}
Hence $[M_n(P),P]\leq M_{n-1}(P)$.

\noindent (c) is obvious.

\noindent (d) By the Hall-Petrescu formula (see \cite[Theorem 9.4]{hupp1:grp}) one has
\begin{equation}
\label{eq:Md}
(xy)^p\equiv
x^py^p\pmod{[M_m(P),P]^p[M_m(P),_{(p-1)}P]}.
\end{equation}
Thus if $m\leq p-1$ there is nothing to prove. Let $m>p-1$.
Then $M_m(P)=M_{m-p+1}(P)^{1/p}$, and by (a),
\begin{equation}
\label{eq:Md2}
[M_m(P),_{(p-1)}P]=[M_{m-p+1}(P)^{1/p},_{(p-1)}P]=[M_{m-p+1}(P),P]=M_{m-p}(P).
\end{equation}
Hence, (b) and (c) yield the claim.
\end{proof}

\begin{proof}[Proof of Theorem C]
Let $m_\bullet(P)= \coprod_{n\in \N} M_{n+1}(P)/M_n(P)$.
By Proposition~\ref{prop-1}(d), 
$t_\bullet\colon m_\bullet(P)\to m_\bullet(P)$,
$t_n(xM_{n-1}(P))=x^pM_{n-p}(P)$, is a well-defined, $\F_p$-linear mapping
of graded $\F_p$ vector spaces of degree $1-p$. Thus $m_\bullet(P)$ is a graded $\F_p[t_\bullet]$ module. Moreover, by construction of the series
$(M_n(P))_{n\geq 0}$, one has
\begin{equation}
\label{eq:pfthmC-1}
\Ann(t_\bullet)=m_{(p-1)}(P)=\textstyle{\coprod_{0\leq n\leq p-2}} M_{n+1}(P)/M_n(P).
\end{equation}
For any subgroup $Q$ of $P$ we put
$m_\bullet(Q)=\coprod _{n\in \N}(Q\cap M_{n+1}(P))M_n(P)/M_n(P)$. 
Hence by \eqref{eq:pfthmC-1}, $\Ann(t_\bullet)=m_\bullet(\Omega_1(P))$.
On the other hand one has
\begin{equation}
\label{eq:pfthmC-2}
t_\bullet.M(P)=\textstyle{\coprod_{n\in \N} M_{n+p}(P)^pM_n(P)/M_n(P)}
\end{equation}
and
\begin{equation}
\label{eq:pfthmB-3}
m_\bullet(P^p)=\textstyle{\coprod_{n\in \N} (P^p\cap M_{n+1}(P))M_n(P)/M_n(P)}.
\end{equation}
As $M_{n+p}(P)^p\leq P^p\cap M_{n+1}(P)$, one has
\begin{equation}
\label{eq:pfthmB-4}
|P:P^p|=|m_\bullet(P):m_\bullet(P^p)|\leq |m_\bullet(P):t_\bullet.m_\bullet(P)|=
|\Ann(t_\bullet)|=
|\Omega_1(P)|.
\end{equation} 
Moreover, if
$P^p=\{\,x^p\mid x\in P\,\}$, then $P^p\cap M_{n+1}(P)=M_{n+p}(P)^p$. Hence
equality holds in \eqref{eq:pfthmB-4}.
\end{proof}

\section{Fusion of elementary abelian $p$-subgroups and $F$-isomorphisms of 
mod-$p$ cohomology rings}
\label{s:fus}
Throughout this section we fix a prime number $p$.

\subsection{The mod $p$ cohomology ring of a finite group}
\label{ss:modpcoh}
Recall that an $\N_0$-graded $\F_p$-algebra $A_\bullet$ is called
{\it anti-commutative}, if  for $a\in A_n$, $b\in A_m$, $n,m\geq 0$ one has
\begin{equation}
\label{eq:bas-1}
b\cdot a =(-1)^{n\cdot m} a\cdot b.
\end{equation}
Such an $\F_p$-algebra is called {\it connected}, if $A_0=\F_p.1_A$.
For a finite group $G$ the {\it cohomology algebra}
\begin{equation}
\label{eq:bas-11}
H^\bullet(G)=H^\bullet(G,\F_p)
\end{equation}
is a finitely generated, connected, anti-commutative, $\N_0$-graded $\F_p$-algebra.

Let $\alpha_\bullet\colon A_\bullet\to B_\bullet$ be a homomorphism of
finitely generated, connected, anti-commutative, $\N_0$-graded $\F_p$-algebras.
Then $\alpha_\bullet$ is called an {\it $F$-isomorphism}, if
$\kernel(\alpha_\bullet)$ is nilpotent, and for all $b\in B_n$, there exists
$k\in\N_0$ such that $b^{p^k}\in\image(\alpha_\bullet)$.
If $S$ is a subgroup of the finite group $G$, then
\begin{equation}
\label{eq:bas-3}
\rst^G_S\colon H^\bullet(G)\longrightarrow H^\bullet(S)
\end{equation}
is a homomorphism of
finitely generated, connected, anti-commutative, $\N_0$-graded $\F_p$-algebras.

\subsection{Quillen stratification}
\label{ss:strat}
Let $G$ be a finite group. Let $\ca{C}_G$ denote the category
the objects of which are elementary $p$-abelian subgroups of $G$ and morphisms are given by conjugation, i.e., for $E, E^\prime\in ob(\ca{C}_G)$ one has
\begin{equation}
\label{eq:bas-4}
\mor_G(E,E^\prime)=\{\,i_g\colon E\to E^\prime\mid g\in G,\ g\,E\,g^{-1}\leq E^\prime\,\},
\end{equation}
where $i_g(e)= g\, e\, g^{-1}$, $e\in E$. Then
\begin{equation}
\label{eq:bas-5}
H^\bullet(\ca{C}_G)=\textstyle{\varprojlim_{\ca{C}_G} H^\bullet(E)}
\end{equation}
is a finitely generated, connected, anti-commutative, $\N_0$-graded $\F_p$-algebra.
Moreover, the restriction maps $\rst^G_E$ yield a map
\begin{equation}
\label{eq:bas-6}
q_G=\textstyle{\prod_{E\in ob(\ca{C}_G)}\rst^G_E}\colon H^\bullet(G)\longrightarrow
H^\bullet(\ca{C}_G).
\end{equation}
Although it is quite difficult to determine the explicit structure of the mod $p$ cohomology algebra $H^\bullet(G)$ for a given finite group $G$, one knows its structure up to
$F$-isomorphism thanks to the following remarkable result of D. Quillen which is also known as
Quillen stratification (see \cite[Cor. 5.6.4]{ben:coh2}, \cite{qu:strat}).

\begin{thm}[D. Quillen]
\label{thm:strat}
Let $G$ be a finite group. Then $q_G\colon H^\bullet(G)\to H^\bullet(\ca{C}_G)$
is an $F$-isomorphism.
\end{thm}

\subsection{Quillen stratification and fusion}
\label{ss:stratfus} 
Let $H$ be a subgroup of the finite
group $G$ containing a Sylow $p$-subgroup of $G$ which controls 
fusion of elementary abelian $p$-subgroups. Then the canonical embedding functor
\begin{equation}
\label{eq:embedfun}
j_{H,G}\colon\ca{C}_H\longrightarrow\ca{C}_G
\end{equation}
is an equivalence of categories. Thus 
$H^\bullet(j_{H,G})\colon H^\bullet(\ca{C}_G)\to H^\bullet(\ca{C}_H)$
is an isomorphism of
connected, anti-commutative, $\N_0$-graded $\F_p$-algebras. Thus from
Theorem~\ref{thm:strat} one concludes the following property
which can be seen as a weak version of G.~Mislin's result (see \cite{mis:modp}).

\begin{prop}
\label{prop:fusF}
Let $H$ be a subgroup of the finite
group $G$ containing a Sylow $p$-subgroup of $G$ which controls 
fusion of elementary abelian $p$-subgroups.
Then $\rst^G_S\colon H^\bullet(G)\to H^\bullet(H)$ is an $F$-isomorphism.
\end{prop}

\begin{proof}
This is a direct consequence of the commutative diagram
\begin{equation}
\label{dia:fusF}
\xymatrix{
H^\bullet(G)\ar[d]_{\rst^G_H}\ar[r]^{q_G}& H^\bullet(\ca{C}_G)\ar[d]^{H^\bullet(j_{H,G})}\\
H^\bullet(H)\ar[r]^{q_H}& H^\bullet(\ca{C}_H)
}
\end{equation}
and the previously mentioned remark.
\end{proof}

\subsection{Quillen's $p$-nilpotency criterion}
\label{ss:pnil}
In \cite{qu:pnil}, D.~Quillen has shown the following.

\begin{thm}[D. Quillen]
\label{thm:qucrit}
Let $G$ be a finite group, let $p$ be odd, and let $P$ be a Sylow $p$-subgroup
of $G$. Then $G$ is $p$-nilpotent if, and only if,
$\rst^G_P\colon H^\bullet(G)\to H^\bullet(P)$ is an $F$-isomorphism.
\end{thm}

\begin{rem}
(a) Let $G$ be a finite group, and let $P$ be a Sylow $p$-subgroup.
By a result of J. Tate (see \cite{tat:nil}), $G$ is $p$-nilpotent if, and only if,
$\rst^G_P\colon H^\bullet(G)\to H^\bullet(P)$ is an isomorphism.
Every isomorphism of finitely generated, connected, anti-commutative,
 $\N_0$-graded $\F_p$-algebras is in particular also an $F$-isomorphism.
Hence Theorem~\ref{thm:qucrit} can be seen as a generalization of
J.~Tate's nilpotency criterion
for $p$ odd.

\noindent
(b) Let $p$ be odd. Let $G$ be a finite group, and let $P\in\Syl_p(G)$.
From Proposition~\ref{prop:fusF} and Theorem~\ref{thm:qucrit} one concludes
that $G$ is $p$-nilpotent if, and only if, $P$ controls fusion 
of elementary abelian $p$-subgroups. 
\end{rem}

\subsection{Finite $p$-central groups of height $k$}
\label{ss:finpcen}
For finite $p$-central groups of height $k$
one has the following fundamental property.

\begin{prop}
\label{prop:fungr}
Let $G$ be a finite $p$-central group
of height $k$, $k\geq 1$, and let $P$ be a Sylow $p$-subgroup of $G$. Then for any elementary $p$-abelian
subgroup $E$ of $G$, one has
\begin{equation}
\label{eq:eqfun1}
G=P.\Cent_G(E).
\end{equation}
\end{prop}

\begin{proof}
The subgroup $\Omega_1(G,p)$ is nilpotent of class at most $k$, and therefore
a finite $p$-group. Let $p^e$ be the exponent of $\Omega_1(G,p)$.
By the Hall-Petrescu collection formula (see \cite[Theorem 2.1]{gust:omega}), 
for any $y\in \Omega_1(G,p)$ and $x\in G$ one has
\begin{equation}
\label{eq:bas-0}
[y,x^{p^{e+k}}]\in 
\prod_{0\leq i\leq e+k}[\Omega_1(G,p),_{(p^i)}G]^{p^{e+k-i}}=\{1\}.
\end{equation}
In particular, $G^{p^{e+k}}\leq \Cent_G(\Omega_1(G))$.
Moreover, for any Sylow $p$-subgroup $P$ of $G$ one has
$G=P.G^{p^{e+k}}$. Since every elementary $p$-abelian subgroup $E$
of $G$ is contained in $\Omega_1(G,p)$, this yields the claim.
\end{proof}

\begin{proof}[Proof of Theorem~A]
Let $G$ be a finite $p$-central group of height $k$, $p$ odd.
By Proposition~\ref{prop:fungr}, $P$ controls fusion of elementary abelian
$p$-subgroups. Then Proposition~\ref{prop:fusF} and Theorem~\ref{thm:qucrit}
yield the claim.
\end{proof}

\section{Finite $p$-groups which determine $p$-nilpotency locally}
\label{s:dnl}
Let $p$ be an odd prime number.
For $m\geq 1$ we denote by $\boC_{p^m}$ the cyclic
group of order $p^m$. 
Let $Y_p(1)=\boC_p\wr\boC_p$ denote the regular
wreath product, and let $\beta\colon\boC_p\wr\boC_p\to\boC_p$
denote the canonical homomorphism with elementary abelian kernel.
By $Y_p(m)$, $m\geq 1$, we denote the finite $p$-group
which is the pull-back of the diagram
\begin{equation}
\label{eq:defX}
\xymatrix{
\boC_{p^m}\ar[r]&\boC_p\\
Y_p(m)\ar@{-->}[r]\ar@{-->}[u]&\boC_p\wr\boC_p\ar[u]_\beta
}
\end{equation}
In \cite{wei:dnl} it was shown that every finite $p$-group $P$ not containing
subgroups isomorphic to $Y_p(m)$, $m\geq 1$, determines $p$-nilpotency locally.
As a consequence we obtain the following result.

\begin{thm}
\label{thm:dnl}
Let $p$ be odd, and let $P$ be a finite $p$-central $p$-group of height $p-1$.
Then $P$ does not contain subgroups isomorphic to $Y_p(m)$ for some $m\geq 1$.
In particular, $P$ determines $p$-nilpotency locally.
\end{thm} 

\begin{proof}
Suppose $P$ contains a subgroup isomorphic to $Y_p(m)$, $m\geq 1$.
This implies that $Y_p(m)$ is $p$-central of height $p-1$, a contradiction.
\end{proof}

\section{Semi-simple modules for finite $p$-soluble groups}
\label{s:psol}
Let $p$ be a prime number, and let
$\F_q$, $q=p^f$, denote the finite field with $q$ elements.
Then there exists a unique unramified extension $\Q_q/\Q_p$
of the $p$-adic numbers of degree $f$. Moreover,
$\Z_q=\incl_{\Q_q}(\Z_q)$ is a complete discrete valuation domain
with maximal ideal $p.\Z_q$ and residue field isomorphic to $\F_q$.

For a finite group $G$ we will call a finitely generated left $\Z_q[G]$-module $L$ a
left {\it $\Z_q[G]$-lattice}, if it is also a free $\Z_q$-module.

The purpose of this section is to translate
I.~M.~Isaacs version (see \cite{isa:psol1},~\cite{isa:psol2})
of the Fong-Swan-Rukolaine theorem (see \cite[Thm.~22.1]{currei:vol1}) in a module theoretic way
(see Thm.~\ref{thm:isaacs}). Using results of P.~H.~Tiep and A.~E.~Zalesskii
one can then
describe the size of a Jordan block of an element of order $p$
in a $p$-soluble group $G$ on any finite semi-simple left $\F_q[G]$-module
(see Cor.~\ref{cor:jor}).
This description will allow us to establish a sufficient criterion for the normality
of a Sylow $p$-subgroup in a finite $p$-soluble group provided $p$ is odd
(see Lemma~\ref{lemma:crit}).

\subsection{The Fong-Swan-Rukolaine theorem}
\label{ss:isaacs}
The following theorem can be seen as a module theoretic version of
I.~M.~Isaacs version 
of the Fong-Swan-Rukolaine theorem.

\begin{thm}
\label{thm:isaacs}
Let $G$ be a finite $p$-soluble group, let $\bE$ be a finite field of characteristic $p$
and let $M$ be an
absolutely irreducible left $\bE[G]$-module.
Then there exists a finite extension $\F_q/\bE$ of $\bE$ and 
a left $\Z_q[G]$-lattice $L$ with the following properties:
\begin{itemize}
\item[(i)] $L/p.L\simeq M\otimes_{\bE}\F_q$;
\item[(ii)] $L\otimes_{\Z_q}\Q_q$ is an absolutely irreducible left
$\Q_q[G]$-module.
\end{itemize}
\end{thm}

\begin{proof}
Let $\bbE$ denote an algebraic closure of $\bE$, let $\bQ_p$ denote an
algebraic closure of $\Q_p$, and let
$\Q_p^{\un}$ denote the maximal unramified extension of $\Q_p$ in $\bQ_p$.
Let $\bchi_M$ denote the Brauer character of $M\otimes_{\bE}{\bbE}$.
By \cite{isa:psol1} and \cite{isa:psol2}, there exists an irreducible left $\bQ_p[G]$-module
$V$ such that the restriction of its character to $p$-regular conjugacy classes of $G$ agrees 
with $\bchi_M$, and all its character values are contained in $\Q_p^{\un}$.
Then by \cite[Lemma~3.3]{tiza:long}, there exists  a finite unramified
extension $\Q_q/\Q_p$ of $\Q_p$ and an absolutely irreducible 
left $\Q_q[G]$-module $W$ such that 
$V\simeq W\otimes_{\Q_q}\bQ_p$.  
We may assume that $\F_q$ contains $\bE$.
Let $L< W$ be a full left $\Z_q[G]$-lattice contained in $W$.
Then by construction $L/p.L\simeq M\otimes_{\F_q}\bE$, and
$L\otimes_{\Z_q}\Q_q\simeq W$. This yields the claim.
\end{proof}

From \cite[Lemma~3.1]{tiza:long} one concludes the following corollary.

\begin{cor}
\label{cor:jor}
Let $G$ be a finite $p$-soluble group, let $M$ be a finite
semi-simple $\F_q[G]$-module, and let $\phi\colon G\to\End_{\F_q}(M)$ be the
representation associated to $M$. Then for every element $g\in G$ of order $p$
the size of a Jordan block of $\phi(g)$ is either $1$, $p-1$ or $p$.
\end{cor}

Corollary~\ref{cor:jor} can be translated into a group theoretic context as follows:

\begin{lemma}
\label{lemma:crit}
Let $G$ be a finite $p$-soluble group, $p\not=2$, and let $P\in\Syl_p(G)$.
Let $M$ be a finitely generated $\F_p[G]$-module,
and let $\phi_M\colon G\to\GL_{\F_p}(M)$ denote the associated 
representation. Assume further that
\begin{itemize}
\item[(i)] $\kernel(\phi_M)\leq O_p(G)$;
\item[(ii)] for all $g\in P$, $(\phi_M(g)-\iid_M)^{p-2}=0$.
\end{itemize}
Then $P=\kernel(\phi_M)$; in particular, $P$ is normal in $G$.
\end{lemma}

\begin{proof}
Let $\pi\colon M\to M^{\ssl}$ denote the semi-simplification of the $\F_p[G]$-module
$M$, and let 
$\phi_{M^{\ssl}}\colon G\to\GL_{\F_p}(M^{\ssl})$ denote the associated $\F_p$-linear representation.
By~(i), $\kernel(\phi_M)$ is a normal $p$-subgroup.
Let $\tphi_M\colon\F_p[G]\to\End_{\F_p}(M)$ denote the associated
homomorphism of $\F_p$-algebras. Then
\begin{equation}
\label{eq:jac}
\kernel(\phi_{M^{\ssl}})/\kernel(\phi_M)\leq \iid_M+\Jac(\image(\tphi_M)),
\end{equation}
where $\Jac(\argu)$ denotes the Jacobson radical.
In particular, $\kernel(\phi_{M^{\ssl}})/\kernel(\phi_M)$ is a $p$-group, and thus
$\kernel(\phi_{M^{\ssl}})$ is a normal $p$-subgroup of $G$.

Let $g\in P^\circ=\phi_{M^{\ssl}}(P)$
be an element satisfying $g^p=1$.
Since $\phi_{M^{\ssl}}(G)$ is $p$-soluble and as
$M^{\ssl}$ is a semi-simple left $\F_p[\phi_{M^{\ssl}}(G)]$-module, 
Corollary~\ref{cor:jor} implies that a Jordan block
of $g$ must have length $1$, $p-1$ or $p$.
Hypothesis (ii) yields that $(g-1)^{p-2}=0\in\End_{\F_p}(M^{\ssl})$. Thus every Jordan block of $g$
must have size $1$; i.e., $g=1$. Therefore, $P^\circ=\{1\}$, and this yields the claim.
\end{proof}

\subsection{The sandwich theorem}
\label{ss:sand}
Let $G$ be a finite group. By $O_{p^\prime,p}(G)$ we denote the normal 
subgroup of $G$ containing $O_{p^\prime}(G)$ and satisfying
\begin{equation}
\label{eq:opp}
O_{p^\prime,p}(G)/O_{p^\prime}(G)=O_p(G/O_{p^\prime}(G)).
\end{equation}
Similarly, we denote by $O_{p^\prime,p,p^\prime}(G)$ the normal subgroup of
$G$ containing $O_{p^\prime,p}(G)$ such that
\begin{equation}
\label{eq:oppp}
O_{p^\prime,p,p^\prime}(G)/O_{p^\prime,p}(G)=O_{p^\prime}(G/O_{p^\prime,p}(G)).
\end{equation}
The finite group $G$ will be called a {\it $p^\prime p p^\prime$-sandwich group},
if $G=O_{p^\prime,p,p^\prime}(G)$. Note that
a $p^\prime p p^\prime$-sandwich group is $p$-soluble. One has the following
property.

\begin{prop}
\label{prop:ppp}
Let $G$ be a finite $p^\prime p p^\prime$-sandwich group, and let $P\in\Syl_p(G)$.
Then $N_G(P)$ controls $p$-fusion in $G$.
\end{prop}

\begin{proof}
Let $Q=O_{p^\prime,p}(G)$ and $M=O_{p^\prime}(G)$; in particular, $P\leq Q$.
The Frattini argument yields that $G=Q.N_G(P)=M.N_G(P)$. Put
$M_\circ=M\cap N_G(P)$. Then one has a commutative diagram of finite groups
\begin{equation}
\label{dia:grdia}
\xymatrix{
N_G(P)\ar[rd]_\beta\ar[0,2]^\alpha&& G\ar[ld]^\gamma\\
&N_G(P)/M_\circ&
}
\end{equation}
and thus also a commutative diagram of homomorphism of $\F_p$-algebras
\begin{equation}
\label{dia:Hdia}
\xymatrix{
H^\bullet(N_G(P),\F_p)&& H^\bullet(G,\F_p)\ar[0,-2]_{H^\bullet(\alpha)}\\
&H^\bullet(N_G(P)/M_\circ,\F_p)\ar[ul]^{H^\bullet(\beta)}\ar[ur]_{H^\bullet(\gamma)}&
}
\end{equation}
Since $H^\bullet(\beta)$ and $H^\bullet(\gamma)$ are inflation mappings with respect
to normal subgroups of $p^\prime$-order, $H^\bullet(\beta)$ and $H^\bullet(\gamma)$
are isomorphisms. Hence $H^\bullet(\alpha)=\rst^G_{N_G(P)}$ is an isomorphism.
Thus by G.~Mislin's theorem (see \cite{mis:modp}), $N_G(P)$ controls $p$-fusion in $G$.
\end{proof}

The following lemma will turn out to be useful for our purpose.

\begin{lemma}
\label{lemma:cenin}
Let $p$ be an odd prime, and let $G$ be a finite $p$-soluble group.
Assume further that $O_{p^\prime}(G)=\{1\}$, $O_p(G)\not=\{1\}$,
and that $\Omega_1(O_p(G))$ is of exponent $p$. Then
$\Cent_G(\Omega_1(O_p(G))\leq O_p(G)$.
\end{lemma}

\begin{proof}
Put $\Omega=\Omega_1(O_p(G))$ and $C=\Cent_G(\Omega)$.
Then $O_p(C)\leq O_p(G)$,
and therefore $\Omega_1(O_p(C))\leq\Omega$.
Thus $\Omega_1(O_p(C))\leq\Zen(\Omega)$, and $O_p(C)$ is a non-trivial
$p$-central normal subgroup of $G$.
Let $M$ be the characteristic subgroup of $C$ containing $O_p(C)$ such that 
$M/O_p(C)=O_{p^\prime}(C/O_p(C))$. Let $H\leq M$ be a $p^\prime$-Hall
subgroup of $M$. Then $H$ is acting trivially on $\Omega_1(O_p(C))$,
and thus on $O_p(C)$ (see \cite[Kap.~IV, Satz~5.12]{hupp1:grp}). 
In particular, $[H,O_p(C)]=\{1\}$, and therefore
$H=O_{p^\prime}(M)$. Hence, as $O_{p^\prime}(G)=\{1\}$, $H=\{1\}$,
and $C=O_p(C)$. This yields the claim.
\end{proof}

\begin{thm}
\label{thm:sand}
Let $p$ be odd, and let $G$ be a finite $p$-soluble group containing a Sylow $p$-subgroup $P\in\Syl_p(G)$
which is $p$-central of height $p-2$. Then $P.O_{p^\prime}(G)$ is normal in $G$.
In particular, $G$ is a $p^\prime p p^\prime$-sandwich group.
\end{thm}

\begin{proof}
We may assume that $O_{p^\prime}(G)=\{1\}$ and $O_p(G)\not=\{1\}$.
Let $\Omega=\Omega_1(O_p(G))$. Since $P$ is $p$-central of height $p-2$, 
$\Omega$ is of exponent $p$.
Thus by Lemma~\ref{lemma:cenin}, one has $\Cent_G(\Omega)\leq O_p(G)$.

Let $M=\gr_\bullet^{(1)}(\Omega)$ denote the graded $\F_p[G]$-module associated to
the descending $p$-central series of $\Omega$ (see \eqref{eq:lamgr}), and let
$\phi_M\colon G\longrightarrow\GL_{\F_p}(M)$
denote the associated representation for $G$. 
Let $g\in \kernel(\phi(M))$ be an element of order
co-prime to $p$. Then $g\in\Cent_G(\Omega)\leq O_p(G)$.
Hence $g=1$, and $\kernel(\phi_M)\leq O_p(G)$.
As $P$ is $p$-central of height $p-2$,  one has for $g\in P$ that
$(\phi_M(g)-\iid_M)^{p-2}=0$. Hence $\kernel(\phi_M)=P$
by Lemma~\ref{lemma:crit}. This yields the claim.
\end{proof}

From Theorem~\ref{thm:sand} and Proposition~\ref{prop:ppp} one concludes
the following consequence which can be seen an analogue of 
J.~Th{\'e}venaz's result.

\begin{cor}
\label{cor:psolfus}
Let $p$ be odd, and let $G$ be a finite $p$-solvable group, for which
$P\in\Syl_p(G)$ is $p$-central of height $p-2$.
Then $N_G(P)$ controls $p$-fusion in $G$.
\end{cor}


\providecommand{\bysame}{\leavevmode\hbox to3em{\hrulefill}\thinspace}
\providecommand{\MR}{\relax\ifhmode\unskip\space\fi MR }
\providecommand{\MRhref}[2]{%
  \href{http://www.ams.org/mathscinet-getitem?mr=#1}{#2}
}
\providecommand{\href}[2]{#2}


\end{document}